\documentclass{amsart}
\usepackage{amssymb,amsmath,amsfonts,pictex,graphicx,etex,color,tikz}
\usetikzlibrary{shapes,shapes.arrows,snakes,positioning}

\newcommand{\A}{{\sf a}}
\newcommand{\B}{{\sf b}}
\newcommand{\C}{{\sf c}}
\newcommand{\E}{{\sf e}}
\renewcommand{\P}{{\sf p}}

\newcommand{\V}{{\sf v}}
\newcommand{\U}{{\sf u}}
\newcommand{\W}{{\sf w}}
\newcommand{\X}{{\sf x}}
\newcommand{\Y}{{\sf y}}

\newcommand{\Z}{\mathbb{Z}}
\newcommand{\R}{\mathbb{R}}

\newcommand{\std}{{\sf std}}
\newcommand{\sdot}{\! \cdot \!}
\newcommand{\muL}{\mu\raisebox{-3pt}{\scriptsize{\ensuremath L}}}
\newcommand{\normL}[1]{\| #1 \| \raisebox{-3pt}{\scriptsize{\ensuremath L}}}

\definecolor{thingie}{rgb}{.75,.4,.75}
\definecolor{brown}{rgb}{.545,.27,.075}

\newcommand{\eN}{\mathcal{N}}
\newcommand{\eP}{\mathcal{P}}

\DeclareMathOperator{\Vol}{Vol}
\DeclareMathOperator{\vol}{Vol}
\DeclareMathOperator{\Area}{Area}

\DeclareMathOperator{\CHull}{CHull}

\newcommand{\eadd}{\overset{\scriptscriptstyle +}{\asymp}}

\theoremstyle{plain}
\newtheorem{theorem}{Theorem}[section]

\newtheorem{prop}[theorem]{Proposition}
\newtheorem{lemma}[theorem]{Lemma}
\newtheorem{corollary}[theorem]{Corollary}
\newtheorem{remark}[theorem]{Remark}

\newtheorem{defn}[theorem]{Definition}

\newtheorem*{sphere-asymp}{Theorem \ref{sphere-asymp}}
\newtheorem*{ball-asymp}{Corollary \ref{ball-asymp}}
\newtheorem*{sphereball}{Theorem \ref{sphereball}}

\begin{document}

\title{The geometry of spheres in free abelian groups}
\author[Duchin Leli\`evre Mooney]{Moon Duchin, Samuel Leli\`evre, and Christopher Mooney}
\date{\today}

\begin{abstract}
We study word metrics on $\Z^d$ by developing tools that are fine enough to measure dependence on the generating set.
We obtain counting and distribution results for the words of length $n$.
With this, we show that counting measure on spheres always converges to a limit measure on a limit shape (strongly, in an appropriate sense).
The existence of a limit measure is quite strong---even virtually abelian groups need not satisfy these kinds 
of asymptotic formulas.
Using the limit measure, we can reduce probabilistic questions about word metrics  to problems in convex geometry of
Euclidean space.  
As an application, we give asymptotics for the spherical growth function with respect to any generating set, as well as statistics for other ``size-like" functions.
\end{abstract}

\maketitle

\section{Introduction}

Suppose one wants to study the density of group elements that have a certain property, or
the average value of some statistic in a group.  If the group is finitely generated by a
set $S$, then there is an associated word metric on the group that measures how far an
element is from the identity, and the ball of radius $n$ is a finite set.  Arguably the
most natural approach to a density question is to put counting measure on the ball of
radius $n$, measure the proportion of those points with the desired property, and let $n$
tend to infinity.

Given a group $G$ with a fixed finite generating set $S$ (say symmetrized, so that
$S=S^{-1}$), let $S_n$ denote the sphere of radius $n$ in the Cayley graph, which is just
the set of group elements whose distance from the identity in the word metric is exactly
$n$; that is, they are group elements whose minimal spelling has $n$ letters, or which
are reached by geodesics of length $n$ based at the identity.  Likewise, $B_n$ is the
(closed) ball of radius $n$.  Then a reasonable way to consider the density in $G$ of a
property with respect to a generating set is to find the expectation over large balls
$B_n$.  Furthermore, one might be interested in understanding the frequency of a property among
long words, which amounts to finding the expectation over
large spheres $S_n$---a strictly harder problem, as we will demonstrate.

More generally, we will study the averages $\frac{1}{|B_n|}
\sum_{\X\in B_n} f(\X)$ and  $\frac{1}{|S_n|} \sum_{\X\in S_n} f(\X)$, not just for
characteristic functions.  For some functions we will find that these averages grow on the
order of $n^k$, in which case we normalize the average and seek a limit, or the
coefficient of $n^k$.

We will show that averages for ``size-like" functions
over spheres in $(\Z^d,S)$ must exist; further, the averages can be reduced to
integrals on convex polyhedra in Euclidean space, with respect to an appropriate
geometrically defined measure.  One of the themes will be to study statistics that are a
priori dependent on the choice of generating set.  In some cases, we will be able to
quantify the extent of the dependence; in other cases, we will find that there is no
dependence.

Recall that a function $g:\R^d\to\R$ is called {\em homogeneous (of order $k$)} if
$g(a\X)=a^kg(\X)$ for $a\ge 0$.  Let us call a function $f:\Z^d \to \R$ {\em coarsely
homogeneous} if there is some homogeneous function $g:\R^d\to\R$ such that $f \eadd g$,
meaning that $|g(\X)-f(\X)|$ is uniformly bounded over $\X\in\Z^d$.  We say that $f$ is
{\em asymptotically homogeneous} if there is some homogeneous function such that $f\sim
g$, meaning that the ratio $f(\X)/g(\X)\to 1$ as $\X\to\infty$.
(Here the notation $\X\to\infty$ means that sequences leave all compact sets.)
In particular, coarsely homogeneous implies asymptotically homogeneous when $k\ge 1$ and $g$ is nonzero.

For any free abelian group $\Z^d$ with any finite generating set $S$, let $Q$ be the
convex hull of the points corresponding to $S$ in $\R^d$, and let $L$ be its boundary
polyhedron, and let $\hat A$ denote the cone from $A\subseteq L$ to the origin, so that
$Q=\hat L$.  For $A\subseteq L$ define the {\em cone measure} by
$$\mu(A)=\muL(A)= \frac{\vol(\hat A)}{\vol(Q)}.$$
This is the Euclidean volume of the cone from $A$ to the origin normalized by the volume
of $Q$, so that $\mu=\muL$ is a Borel probability measure on $L$ (and we will suppress
$L$ from the notation when it is understood).  As we will discuss below, it is not hard to show that 
$S_n$ looks more and more like the the dilate $nL$, or in other terms, that
$\frac 1n S_n\to
L$ as a Gromov-Hausdorff limit.  Our main contribution is to establish that counting
measure on spheres limits to the cone measure on $L$ in an appropriate sense to carry out
averaging operations.%
\footnote{We note that working in the asymptotic cone, $\R^d$, would be a substitute for the rescaling by dilations;  
furthermore,  this suggests a  
natural generalization of these questions to other groups.  To do this, one would need a theory of ultralimits of measures;
in logic, this goes by the name of Loeb measure.  These ideas have not yet been imported into geometric group theory,
a translation that seems as though it would be quite fruitful.}

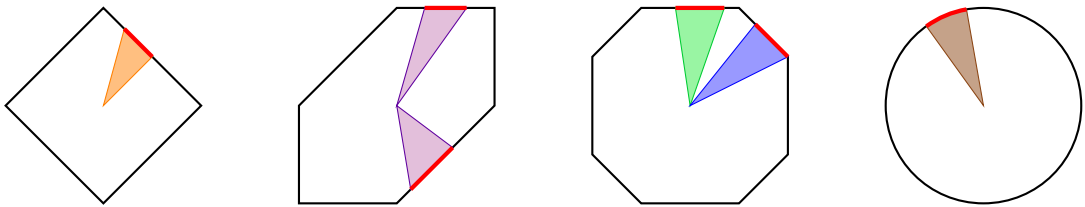
\begin{figure}[ht]
\hspace*{-.3in}
\begin{tikzpicture}
\begin{scope}[scale=1.3]
\draw[black, thick](0,-1)--(1,0)--(0,1)--(-1,0)--cycle;
\draw[orange,fill=orange!50](0,0)--(3/14,11/14)--(7/14,7/14)--cycle;
\draw[ultra thick,red](3/14,11/14)--(7/14,7/14);

\begin{scope}[xshift=3cm]
\draw[black, thick](-1,-1)--(0,-1)--(1,0)--(1,1)--(0,1)--(-1,0)--cycle;
\draw[blue!50!purple,fill=blue!25!purple!25](0,0)--(2/7,1)--(5/7,1)--cycle;
\draw[ultra thick,red](2/7,1)--(5/7,1);
\draw[blue!50!purple,fill=blue!25!purple!25](0,0)--(1/7,-6/7)--(4/7,-3/7)--cycle;
\draw[ultra thick,red](1/7,-6/7)--(4/7,-3/7);
\end{scope}

\begin{scope}[xshift=6cm, scale=1/2]
\draw[black, thick] (2,1)--(1,2)--(-1,2)--(-2,1)--(-2,-1)--(-1,-2)--(1,-2)--(2,-1)--cycle;
\draw[blue!20!green,fill=blue!10!green!40](-.3,2)--(.7,2)--(0,0)--cycle;
\draw[ultra thick,red](-.3,2)--(.7,2);
\draw[blue,fill=blue!40](4/3,5/3)--(2,1)--(0,0)--cycle;
\draw[ultra thick,red](4/3,5/3)--(2,1);
\end{scope}

\begin{scope}[xshift=9cm]
\draw[black, thick](0,0) circle (1);
\draw[brown, fill=brown!50](0,0)-- (100:1) arc (100:880/7:1)--cycle;
\draw[ultra thick,red] (100:1) arc (100:880/7:1);
\end{scope}
\end{scope}

\end{tikzpicture}
\caption{Six arcs are shown in red in this figure, each having cone measure $1/14$; in
other words, all of the colored regions have $1/14$ as much area as the convex body they
are in.  In the square and the hexagon, all sides have equal measure because in each example
the triangles subtended by the sides are mutually congruent.  
On the other hand, for this octagon  generated by the {\em chess-knight} moves $\{(\pm 2,\pm 1),(\pm
1,\pm 2)\}$, the measure of its two types of sides (shown with green and blue) is in
the ratio $4:3$.  Cone measure is defined on any convex, centrally symmetric figure, and
in particular it is uniform on the circle.}
\end{figure}

\begin{theorem}[Sphere averages] \label{sphere-asymp}
For any finite presentation $(\Z^d,S)$ and any function $f:\Z^d\to\R$ asymptotic 
to  $g:\R^d\to\R$ with $g$ homogeneous of order $k$, let $L=\partial\CHull(S)$ and let $\muL$ be cone measure on $L$. 
Then

$$\frac{1}{|S_n|} \sum_{\X\in S_n} f(\X) = (v_{g,S})\sdot n^k + O(n^{k-1}),$$
with coefficient $v_{g,S}:=  \int_L g(\X) \ d\muL(\X)$.
\end{theorem}

Notice that this limit measure $\mu$ is uniform on each face of $L$, in contrast
to the hitting measure for random walks in the word metric, which has a Gaussian distribution.
This is of interest because randomness problems in groups are often approached by studying asymptotics of random 
walks, rather than probabilities with respect to the word metric.  This shows that they are in general quite
different.

To establish Theorem~\ref{sphere-asymp}, we prove a counting lemma relating the distribution of sphere points
to the points of a lattice, and then use Ehrhart polynomials to relate lattice points to volume.
As a consequence of this main theorem, we get the leading term for the average value over balls in the word metric.

\begin{remark}[Ball averages]\label{ball-asymp}  With the same assumptions as above, let $Q=\CHull(S)$.  Then 
$$ \frac{1}{|B_n|} \sum_{\X\in B_n}  f(\X) = (V_{g,S})\sdot n^k + O(n^{k-1}).$$
with coefficient $V_{g,S}:=\int_Q g(\X) \ d\Vol(\X)$.
\end{remark}

Crucially, one may observe that this statement about ball averages can be observed much more easily than the 
sphere averages in the main theorem.
(Just use the fact that $\frac 1n B_n$ becomes  
uniformly distributed in $Q$, and that the error term is counted by points in a region of lower-order volume.)
Going further, one can deduce from ball asymptotics that if 
the counting measures on $\frac 1n S_n$ do converge to a measure on $L$, then 
it must be cone measure:  that is the unique measure on $L$ which, considering the
necessary scaling properties, is compatible with Lebesgue measure on $Q$.  However,
there is no guarantee that the limit in Theorem~\ref{sphere-asymp} exists at all, 
even given the limit statement in Remark~\ref{ball-asymp}.   In the next section, we will give further examples to illustrate that 
the sphere problem is strictly harder than the ball problem.

As a consequence of Theorem~\ref{sphere-asymp}, we can deduce a distribution result:  
the ball average for size-like functions compares
to the sphere average by a simple multiplicative factor which is independent of the 
generating set.  This factor depends only on the dimension $d$ and the
growth order $k$.

\begin{theorem}[Spheres versus balls]\label{sphereball}
For any function $f:\Z^d\to \R$ that is asymptotically homogeneous of order $k$, 
$$\lim_{n\to\infty} \frac{1}{|B_n|} \sum_{\X\in B_n} \frac 1{n^k} f(\X) 
=\left(\frac{d}{d+k}\right)   \lim_{n\to\infty} \frac{1}{|S_n|} \sum_{\X\in S_n} \frac 1{n^k} f(\X) .$$

That is, the coefficients of growth for sphere averages and ball averages are related by the simple expression
$V_{g,S}= \left(\frac{d}{d+k}\right) v_{g,S}$.
\end{theorem}

To put this result in context, we remark that there are three situations in which it is clear that the sphere-average should
equal the ball-average.  
 One case is that of any function averaged over a group with exponential growth, where almost all of the points on the ball will be concentrated on its boundary sphere.
Alternately, for any function averaged over $\Z^d$,  the points in the ball
again become increasingly concentrated in the boundary as $d\to\infty$.  
Finally, sphere-averages clearly equal ball-averages for those size-like functions on $\Z^d$ with 
 $k = 0$ (so $f$ is close to a scale-invariant function).   These last two cases provide a plausibility check on the 
$\frac{d}{d+k}$ factor in this theorem.

Next, we can apply Theorem~\ref{sphere-asymp} to reduce problems about the asymptotic study of the geometry of
Cayley graphs for $\Z^d$ to problems in convex geometry.  

We find a supply of examples of asymptotically homogeneous functions from considering
distance in the word metric.
Using the standard embedding of $\Z^d$ in $\R^d$ as the integer lattice, the word metric on the
Cayley graph for $(\Z^d,S)$ is within bounded distance of a norm on $\R^d$, namely the
norm induced by the convex, centrally symmetric polyhedron $L$.  We will denote this
Minkowski norm by $\normL{\X}$ and recall that it is defined as the unique norm for which
$L$ is the unit sphere.  Then it is a basic fact  (Lemma~\ref{bounded-diff} below) 
that there is a uniform bound $K$ such that
$$\normL{\X} \  \le  \  |\X|  \ \le \   \normL{\X} + K$$
for all $\X\in \Z^d$, where $|\cdot |$ is the word length in the Cayley graph and $K$ 
is the largest distance in the word metric from the identity to any lattice point in
$Q$.  
Burago proved more generally that periodic metrics on $\R^d$
are at bounded distance from norms in \cite{burago}; we give a hands-on proof with the optimal constant for 
word metrics here.
This ensures that distance in the Cayley graph for any finite generating set can be regarded as a coarsely 
homogeneous function.  It is also immediate that $f(x)=|\X|^p$ is
asymptotically homogeneous of order $p$, and it follows from this that the $p$th moments
of a word metric---expected position in a large ball, variance and standard deviation, skewness, and so on---are 
in a sense independent of the choice of finite generating set (discussed below).  As another application of our counting
results, we find asymptotics for the growth function $\beta(n)$ and spherical growth function $\sigma(n)$---the number of 
words of length up to $n$ and exactly $n$, respectively---for $\Z^d$
with an arbitrary generating set.

\begin{theorem}[Growth functions]\label{growth}  
Fixing $(\Z^d,S)$ as above, let $\beta(n)=\# B_n$ and $\sigma(n)=\# S_n$ be
the growth function and the spherical growth function, respectively.  Then 
$$\beta(n)= Vn^d + O(n^{d-1}) \ ; \hspace{.5in}
\sigma(n) = (d\sdot V)n^{d-1} + O(n^{d-2}) \ ;$$
where $V=\vol(Q)$ and $d$ is the dimension.
\end{theorem}

There were previously known results in extremely special cases, as surveyed in \cite[VI.2]{delaharpe},
but this settles the first-order  growth asymptotics for $\Z^d$ completely.  These formulas look deceptively
simple, but some obvious possible generalizations are not true:  
in the next section, we show an example of 
a virtually abelian group which does not satisfy asymptotic growth formulas of this form.

\subsection*{Acknowledgments}

We thank Alex Eskin and Ralf Spatzier.  The first author is partially supported by NSF grant DMS-0906086, and the 
third author is partially supported by NSF grant RTG-0602191.

\section{The sphere problem}

It is a problem of extremely classical interest to count the lattice points in balls in metric spaces.
In one well-known form, this is the Gauss circle problem:
where $R(n)$ is the number of standard lattice points in the round disk of radius $n$ centered
at the origin in the plane, Gauss first proved that 
$$R(n)=\pi n^2 + O(n).$$
Given strong enough estimates for functions on balls, one can derive estimates for 
the annular regions $\Delta_n=B_n\setminus B_{n-1}$ defined as differences of successive balls.
From Gauss's estimate in the circle problem we get $R(n)-R(n-1)=2\pi n + O(n)$, which is vacuous:  the estimate tells
us nothing about lattice points in the annulus.  Historically, the next progress on the circle problem was
by Sierpinski, who proved in 1906 that 
$$R(n)=\pi n^2 + O(n^{2/3}),$$
from which we get $R(n)-R(n-1)=2\pi n + O(n^{2/3})$, which is a nontrivial estimate for annuli.
Finding the optimal error term for the circle problem (conjecturally $n^{1/2 + \epsilon}$) is a deep problem 
with ties to the Riemann hypothesis.
Sphere asymptotics for word metrics can be regarded as a group-theoretic version of the Gauss circle problem, 
because for integer-valued metrics such as word metrics, the annular region $B_n\setminus B_{n-1}$  is precisely the sphere $S_n$.
(And furthermore, the counting results rely crucially on enumerating integer points in geometric annuli $\Delta_nL=nQ\setminus(n-1)Q$.)
This work studies counting and distribution problems in all finitely generated word metrics on free abelian groups.
Though we note that  our error term in this case
is already of optimal order, so the problem is certainly not as rich!

There are several issues that should be clarified at the outset.
First, functions that can be averaged over balls do not necessarily admit well-defined averages over spheres and annuli.
Furthermore, the spherical estimates that we obtain in this paper are strictly better than the difference of ball estimates, even though 
the error order is optimal in the ball estimates.  
In this section, we present simple illustrations of some of the subtleties.

Here is a straightforward example in lattice-point counting to explicitly illustrate this issue.
Take $Q$ to be the unit square $[-\frac 12, \frac 12]\times[-\frac 12,\frac 12]$, $L$ its boundary,
and $\|\cdot\|_L$ the corresponding Minkowski norm on $\R^2$
(the norm whose unit circle is $L$, which in this case is half of the sup norm).  
The ball of radius $n$ with respect to
$\|\cdot\|_L$ is the dilate $nQ$ of $Q$ by $n$.
Since these balls are getting ``fatter'' as $n$ gets large, it is straightforward to count the number of lattice
points lying inside $nQ$ to first order:
\[
	\#\Z^2\cap nQ=n^2\Area(Q)+O(n)=n^2 + O(n).
\]
However, it does not follow that the number of lattice points in the annuli equals $2n+O(1)$; 
instead, we get an oscillating sequence.
Letting $\Delta_nL=nQ\setminus(n-1)Q$, we find that whenever $n$ is odd, $\Delta_nL$ contains
no lattice points at all.  But when $n$ is even, it contains $4n$ lattice points.  So the sequence $\#\Z^2\cap\Delta_nL$ is oscillating
between $4n$ and $0$.  Correspondingly, there is no well-defined coefficient of $n$ in the lattice-point count $n^2+O(n)$ 
given above.
Thus, the indicator function of $\Z^2$ has a well-defined average over the ball $nQ$ but not over the annular 
region which is the difference of two balls.  The same phenomenon, that ball averages are well-defined but sphere averages 
are not, can also be observed in groups with complicated growth functions, as in Cannon's example discussed below.

In this paper, we will deal not only with geometric indicator functions, but also for averaging of more general functions, 
which can sometimes pose analogous difficulties.
To take a number-theoretic example, one can study random properties of the integers by choosing
uniformly over $\{1,2,\ldots, n\}$ and letting $n\to\infty$.  This is what is meant by
classical statements of analytic number theory (see \cite{hardy-wright}) such as
\begin{center}
{\em The probability that two integers are relatively prime is $6/\pi^2$}.
\end{center}

We can express 
the probability of relative primality in terms of the Euler phi function $\phi(n)$;  here again,  the order oscillates
(that is, $\frac 1n \phi(n)$ has multiple accumulation points---one for the primes and another for the powers of two, for instance), while the average order 
$\frac 1{n^2}[\phi(1)+\cdots + \phi(n)] $ converges.  With respect to a Cayley graph for 
$\Z^2$, this means
that sphere-averages do not exist, whereas ball-averages tend to $6/\pi^2$.

Let us draw the contrast between this situation and that of the size-like functions covered by Theorem~\ref{sphere-asymp}.
In this  example, the function $f$ that is the indicator for the relative
primality of the coordinates is coarsely homogeneous of order zero, but not asymptotically 
homogeneous, so Theorem~\ref{sphere-asymp} does not apply.  Furthermore, for size-like functions, unlike for 
relative primality, the growth order ($k$) of the sphere-average is the same as for the 
ball-average.  But for size-like functions on groups, we will not only achieve overall counting results, but also distribution results 
which count points in every direction, as illustrated in Figure~\ref{distribution}.

\begin{figure}[ht]
\hspace*{-.8in}\begin{tikzpicture}[scale=.85]
\def\background{\draw [pink] (3,0)--(0,-3)--(-3,0)--(0,3)--cycle;
\draw [blue,fill=blue!20,opacity=.3] (0,0)--(60:3) arc (60:70:3)--cycle;}

{ 
\background
 \begin{scope} [scale=1/2] 
  \foreach\b in {0,90,180,270}
  {\begin{scope}[rotate=\b]
    \foreach \a in {(6,0), (1,0)}
      \fill \a circle (2pt*1.00000000000000);
  \end{scope}}
  \end{scope}
}

 {  
 \begin{scope}[xshift=6.5cm]
\background
 \begin{scope} [scale=1/12] 
  \foreach\b in {0,90,180,270}
  {\begin{scope}[rotate=\b]
    \foreach \a in {(36,0), (31,0), (30,6), (30,1), (29,0), (26,0), (25,6), (25,1), (24,12), (24,7), (23,6), (24,5), (24,2), (23,1), (22,0), (21,0), (20,6), (20,1), (19,12), (19,7), (19,5), (19,2), (18,18), (18,13), (17,12), (18,11), (18,8), (17,7), (16,6), (17,5), (18,4), (18,3), (17,2), (16,1), (15,6), (15,1), (14,12), (14,7), (14,5), (14,2), (13,18), (13,13), (13,11), (13,8), (13,4), (13,3), (12,24), (12,19), (11,18), (12,17), (12,14), (11,13), (10,12), (11,11), (12,10), (12,9), (11,8), (10,7), (10,5), (11,4), (11,3), (10,2), (9,12), (9,7), (9,5), (9,2), (8,18), (8,13), (8,11), (8,8), (8,4), (8,3), (7,24), (7,19), (7,17), (7,14), (7,10), (7,9), (6,30), (6,25), (5,24), (6,23), (6,20), (5,19), (4,18), (5,17), (6,16), (6,15), (5,14), (4,13), (4,11), (5,10), (5,9), (4,8), (4,4), (4,3), (3,18), (3,13), (3,11), (3,8), (3,4), (3,3), (2,24), (2,19), (2,17), (2,14), (2,10), (2,9), (1,30), (1,25), (1,23), (1,20), (1,16), (1,15)}
      \fill \a circle (2pt*2.44948974278318);
  \end{scope}}
  \end{scope}\end{scope}
}

{\begin{scope}[xshift=13cm]
\background
  \begin{scope} [scale=1/40] 
  \foreach\b in {0,90,180,270}
  {\begin{scope}[rotate=\b]
    \foreach \a in {(120,0), (115,0), (114,6), (114,1), (113,0), (110,0), (109,6), (109,1), (108,12), (108,7), (107,6), (108,5), (108,2), (107,1), (106,0), (105,0), (104,6), (104,1), (103,12), (103,7), (103,5), (103,2), (102,18), (102,13), (101,12), (102,11), (102,8), (101,7), (100,6), (101,5), (102,4), (102,3), (101,2), (100,1), (99,6), (99,1), (98,12), (98,7), (98,5), (98,2), (97,18), (97,13), (97,11), (97,8), (97,4), (97,3), (96,24), (96,19), (95,18), (96,17), (96,14), (95,13), (94,12), (95,11), (96,10), (96,9), (95,8), (94,7), (94,5), (95,4), (95,3), (94,2), (93,12), (93,7), (93,5), (93,2), (92,18), (92,13), (92,11), (92,8), (92,4), (92,3), (91,24), (91,19), (91,17), (91,14), (91,10), (91,9), (90,30), (90,25), (89,24), (90,23), (90,20), (89,19), (88,18), (89,17), (90,16), (90,15), (89,14), (88,13), (88,11), (89,10), (89,9), (88,8), (88,4), (88,3), (87,18), (87,13), (87,11), (87,8), (87,4), (87,3), (86,24), (86,19), (86,17), (86,14), (86,10), (86,9), (85,30), (85,25), (85,23), (85,20), (85,16), (85,15), (84,36), (84,31), (83,30), (84,29), (84,26), (83,25), (82,24), (83,23), (84,22), (84,21), (83,20), (82,19), (82,17), (83,16), (83,15), (82,14), (82,10), (82,9), (81,24), (81,19), (81,17), (81,14), (81,10), (81,9), (80,30), (80,25), (80,23), (80,20), (80,16), (80,15), (79,36), (79,31), (79,29), (79,26), (79,22), (79,21), (78,42), (78,37), (77,36), (78,35), (78,32), (77,31), (76,30), (77,29), (78,28), (78,27), (77,26), (76,25), (76,23), (77,22), (77,21), (76,20), (76,16), (76,15), (75,30), (75,25), (75,23), (75,20), (75,16), (75,15), (74,36), (74,31), (74,29), (74,26), (74,22), (74,21), (73,42), (73,37), (73,35), (73,32), (73,28), (73,27), (72,48), (72,43), (71,42), (72,41), (72,38), (71,37), (70,36), (71,35), (72,34), (72,33), (71,32), (70,31), (70,29), (71,28), (71,27), (70,26), (70,22), (70,21), (69,36), (69,31), (69,29), (69,26), (69,22), (69,21), (68,42), (68,37), (68,35), (68,32), (68,28), (68,27), (67,48), (67,43), (67,41), (67,38), (67,34), (67,33), (66,54), (66,49), (65,48), (66,47), (66,44), (65,43), (64,42), (65,41), (66,40), (66,39), (65,38), (64,37), (64,35), (65,34), (65,33), (64,32), (64,28), (64,27), (63,42), (63,37), (63,35), (63,32), (63,28), (63,27), (62,48), (62,43), (62,41), (62,38), (62,34), (62,33), (61,54), (61,49), (61,47), (61,44), (61,40), (61,39), (60,60), (60,55), (59,54), (60,53), (60,50), (59,49), (58,48), (59,47), (60,46), (60,45), (59,44), (58,43), (58,41), (59,40), (59,39), (58,38), (58,34), (58,33), (57,48), (57,43), (57,41), (57,38), (57,34), (57,33), (56,54), (56,49), (56,47), (56,44), (56,40), (56,39), (55,60), (55,55), (55,53), (55,50), (55,46), (55,45), (54,66), (54,61), (53,60), (54,59), (54,56), (53,55), (52,54), (53,53), (54,52), (54,51), (53,50), (52,49), (52,47), (53,46), (53,45), (52,44), (52,40), (52,39), (51,54), (51,49), (51,47), (51,44), (51,40), (51,39), (50,60), (50,55), (50,53), (50,50), (50,46), (50,45), (49,66), (49,61), (49,59), (49,56), (49,52), (49,51), (48,72), (48,67), (47,66), (48,65), (48,62), (47,61), (46,60), (47,59), (48,58), (48,57), (47,56), (46,55), (46,53), (47,52), (47,51), (46,50), (46,46), (46,45), (45,60), (45,55), (45,53), (45,50), (45,46), (45,45), (44,66), (44,61), (44,59), (44,56), (44,52), (44,51), (43,72), (43,67), (43,65), (43,62), (43,58), (43,57), (42,78), (42,73), (41,72), (42,71), (42,68), (41,67), (40,66), (41,65), (42,64), (42,63), (41,62), (40,61), (40,59), (41,58), (41,57), (40,56), (40,52), (40,51), (39,66), (39,61), (39,59), (39,56), (39,52), (39,51), (38,72), (38,67), (38,65), (38,62), (38,58), (38,57), (37,78), (37,73), (37,71), (37,68), (37,64), (37,63), (36,84), (36,79), (35,78), (36,77), (36,74), (35,73), (34,72), (35,71), (36,70), (36,69), (35,68), (34,67), (34,65), (35,64), (35,63), (34,62), (34,58), (34,57), (33,72), (33,67), (33,65), (33,62), (33,58), (33,57), (32,78), (32,73), (32,71), (32,68), (32,64), (32,63), (31,84), (31,79), (31,77), (31,74), (31,70), (31,69), (30,90), (30,85), (29,84), (30,83), (30,80), (29,79), (28,78), (29,77), (30,76), (30,75), (29,74), (28,73), (28,71), (29,70), (29,69), (28,68), (28,64), (28,63), (27,78), (27,73), (27,71), (27,68), (27,64), (27,63), (26,84), (26,79), (26,77), (26,74), (26,70), (26,69), (25,90), (25,85), (25,83), (25,80), (25,76), (25,75), (24,96), (24,91), (23,90), (24,89), (24,86), (23,85), (22,84), (23,83), (24,82), (24,81), (23,80), (22,79), (22,77), (23,76), (23,75), (22,74), (22,70), (22,69), (21,84), (21,79), (21,77), (21,74), (21,70), (21,69), (20,90), (20,85), (20,83), (20,80), (20,76), (20,75), (19,96), (19,91), (19,89), (19,86), (19,82), (19,81), (18,102), (18,97), (17,96), (18,95), (18,92), (17,91), (16,90), (17,89), (18,88), (18,87), (17,86), (16,85), (16,83), (17,82), (17,81), (16,80), (16,76), (16,75), (15,90), (15,85), (15,83), (15,80), (15,76), (15,75), (14,96), (14,91), (14,89), (14,86), (14,82), (14,81), (13,102), (13,97), (13,95), (13,92), (13,88), (13,87), (12,108), (12,103), (11,102), (12,101), (12,98), (11,97), (10,96), (11,95), (12,94), (12,93), (11,92), (10,91), (10,89), (11,88), (11,87), (10,86), (10,82), (10,81), (9,96), (9,91), (9,89), (9,86), (9,82), (9,81), (8,102), (8,97), (8,95), (8,92), (8,88), (8,87), (7,108), (7,103), (7,101), (7,98), (7,94), (7,93), (6,114), (6,109), (5,108), (6,107), (6,104), (5,103), (4,102), (5,101), (6,100), (6,99), (5,98), (4,97), (4,95), (5,94), (5,93), (4,92), (4,88), (4,87), (3,102), (3,97), (3,95), (3,92), (3,88), (3,87), (2,108), (2,103), (2,101), (2,98), (2,94), (2,93), (1,114), (1,109), (1,107), (1,104), (1,100), (1,99)}
      \fill \a circle (2pt*4.47213595499958);
  \end{scope}}
  \end{scope}\end{scope}
}
\end{tikzpicture}
\caption{For the generating set $S=\pm\{ 6\E_1, \E_1, 6\E_2, \E_2\}$, spheres of radius $n=1$, $6$, and $20$ are shown relative
to the dilated limit shape $nL$.  As $n\to\infty$, the proportion of points in $S_n$ that lie in any fixed direction is 
converging. \label{distribution}}
\end{figure}
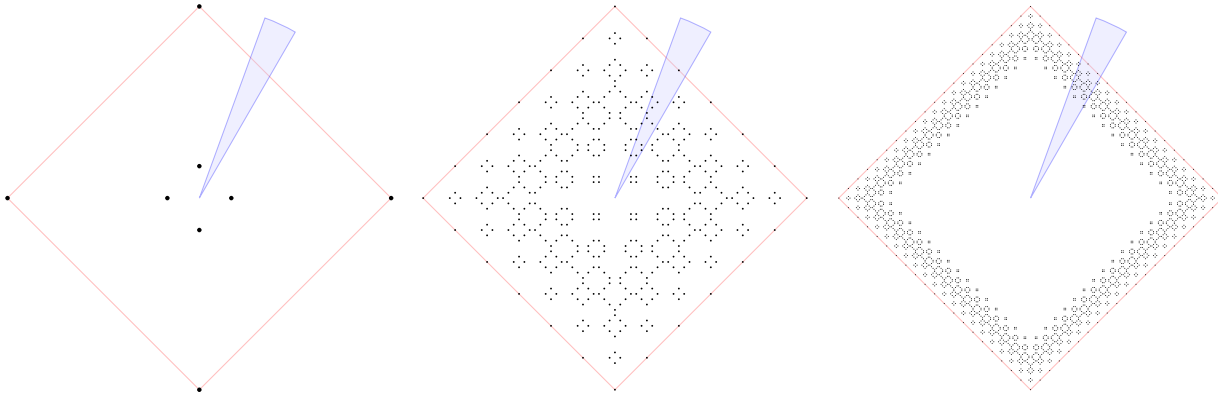

We will find below that 
$$\# S_n \sim \#(\Z^d\cap\Delta_nL) \sim \Vol(\Delta_nL)=dn^{d-1}\sdot\Vol(Q).$$
The first comparison here reduces the sphere counting problem to the lattice point counting problem
(counting integer lattice points in a geometric annular region between balls in the norm $\normL{\cdot}$), and the second comparison
is a solution to the lattice point counting problem.
Counting problems like this can be studied in many other kinds of metric spaces.
In \cite{heis}, for instance, the lattice point counting problem is solved in the Heisenberg group, and partial results are offered on the 
sphere counting problem.

As a final note, 
our findings that $\#S_n=dVn^{d-1}+O(n^{d-2})$ and $\#B_n = Vn^d + O(n^{d-1})$ for any finite generating set
give stronger statements than 
those surveyed in \cite{delaharpe}, where it is stated that $\lim (\#B_n)/n^d =V$ for the standard
generating set and other special cases.   
Beyond this we are not aware of results in the literature prior to the present work concerning asymptotics of
the spherical growth function.  On the other hand, there was a surge of interest in the {\em growth series} of groups
in the 1980s and 1990s.  These are the generating functions for the growth function and spherical growth function,
and they are known to be rational functions for several classes of groups.
  In \cite{benson}, Benson proves that the spherical growth series of any virtually abelian group is a rational function,
and provides an algorithm for computing this precisely.  
But  the fact of rationality does not give information about the asymptotics of the spherical growth function itself.
As a powerful illustration, consider Cannon's example (see \cite[Example VI.A.9]{delaharpe}):  with appropriate generators,  a certain virtually abelian group (the orientation-preserving part of the Euclidean reflection group in the equilateral triangle) 
has a spherical growth function which 
oscillates between two linear functions, and so never even becomes monotone!  Thus growth asymptotics such 
as we compute here are not implied by rationality of the growth series.

\section{The limit metric}
\label{sec:dist}

When considering presentations $(\Z^d,S)$, we will assume throughout that $S$ is 
symmetric, so that $S=-S=S_1$ is also the sphere of radius 1 in the group.  

First let us consider $G=\Z$.  With 
any finite generating set, the spheres of large radius are divided into a positive part
and a negative part, each of uniformly bounded diameter in $\R$.  In particular, if $a$ is the largest positive element in the 
generating set, then the most efficient spelling of a very large integer uses almost exclusively the letter $a$; in the 
language we will develop below, $\pm a$ are the only {\em significant generators}.  
Let $K$ be the smallest value such that the ball of radius $K$ in the word metric
contains all of the integers $-a\le m\le a$ (so that $K$ is a constant depending on $S$).  
The ball of radius $K$ includes all positive integers up to $a$.  The ball of radius $K+1$ then includes all positive integers up to $2a$,
since integers between $a$ and $2a$ can be obtained by adding the generator $a$.
Continuing, we see that, for any $n> K$, the ball of radius $n-1$ includes all positive integers up to $(n-K)a$.
Thus the sphere of radius $n$ is totally contained in the interval $\bigl( (n-K)a, \ na\bigr]$.  
That means that the positive real numbers $\frac 1n S_n$ are contained in the interval $\bigl( a (1-K/n), \ a\bigr]$, 
and therefore $\frac 1n S_n \to \{\pm a\}$.  This depends on the choice of generating set, but in fact only on its
convex hull in $\R$.  

For $G=\Z^d$ the situation is similar.
We will study word-length from a geometric point of view.
Let $S\subset\Z^d$ be a fixed finite set of generators.  
We adopt additive notation, so that every element of \(\Z^d\)
has a representative in the form $\W=\alpha_1\A_1+\alpha_2\A_2+\cdots+ \alpha_r \A_r$ where 
$S=\{\pm \A_1,\ldots,\pm \A_r\}$ and $\alpha_i\in\Z$.
Let $|\W|$ denote the length of $\W$ in the word metric, or the minimal $\sum |\alpha_i|$ 
over all representatives as above.
A spelling is called a \emph{geodesic representative} (or a {\em geodesic spelling}) 
if it realizes this minimum, since these spellings
correspond to geodesic paths in the Cayley graph.

Let $Q$ be the convex hull in $\R^d$ of the generating set  $S$, and let $L$ denote its boundary.
By construction, $L$ is a centrally symmetric convex polyhedron.
Let $\normL{\cdot}$ denote the Minkowski norm on $\R^d$ induced by $L$:  this is the unique norm 
for which $L$ is the unit sphere.  Namely, for $\X\in\R^d$,
$\normL{\X}$ equals the unique $\lambda\ge 0$ such that 
$\X\in \lambda L$.  

For any set $M$, let $\Delta M=\{t\X : t\ge 0, \X\in M\}$ 
be the infinite cone on $M$ from the origin with respect to dilation.
Then let 
$\Delta_k M= \{t \X : k-1< t \le k, \X \in M\}$ be the annular region from the $(k-1)$st to the $k$th
dilation,  so that the cone $\hat M$ from $M$ to the origin is equal to $\Delta_1M\cup\{0\}$.
For $\sigma$ a codimension-1 face of $L$, we will call $\Delta \sigma$ the {\em sector} associated to 
$\sigma$.
$\Vol$ denotes the Lebesgue measure on $\R^d$ (or $\Area$ if $d=2$).
The extreme points of $Q$ are called {\em significant generators.}
These are necessarily elements of the generating set, and it will turn out that many 
of the properties we study in this paper depend only on this subset of $S$.
In particular, we will see shortly that the significant  generators completely determine the averages of size-like functions in the word metric.
The first basic observation is that $L$ encodes the large-scale geometry of the group with this generating
set.

\begin{lemma}
An element $\A\in S$ is on the polytope $L$ if and only if $n\A$ is geodesic for every \(n\in\Z\).
An element $\A\in S$ is an extreme point of $Q$ (or, equivalently, a vertex of $L$) 
if and only if $n\A$ is uniquely geodesic for every \(n\in\Z\).
\end{lemma}

\begin{proof}  We give the proof for $d=2$ for clarity.
Suppose $\A\in S$ is an interior point of $Q$, and suppose it lies in the sector determined by the 
extreme points $\B,\C\in S$.  Then there is a unique positive multiple $\alpha \A$ that lies on the 
segment between $\B$ and $\C$, so $\alpha \A = \beta \B + \gamma \C$, with $\alpha>1$, $\beta+\gamma=1$,
with necessarily rational coefficients since $\A,\B,\C$ have integer coordinates.
But then by clearing common denominators, we have an integer multiple of $\A$ expressed in 
terms of $\B$ and $\C$ with a strictly smaller wordlength.  This shows that $n\A$ cannot be geodesic for
all $n$.

Now suppose $n\A$ is not geodesic for some $n\ge 1$.  Then we can write
$n\A = \sum \alpha_i \A_i$ with $\sum \alpha_i<n$.  Assume further that $n$ is the smallest such value,
so that $\A$ itself does not appear in this spelling.
Then
$$	\normL{\A}  \le \sum \frac{\alpha_i}{n} \left\| \A_i \right\|_L <1, $$
showing that $\A$ is interior to $Q$.

For the word $n\A$, any alternative spelling with the same length expresses $\A$ as a
convex combination of other generators, and such an expression exists if and only if $\A$
is not extreme.

(The proof is the same in arbitrary dimension:  replace $\B,\C$ with the extreme points $\B_1,\ldots,\B_n$ in a cell 
of a triangulation.)
\end{proof}

Now we establish that the word metric limits to a norm, and that they differ
by a bounded additive amount.  
As a consequence, the spheres in the word metric, once normalized, converge to a 
limit shape.
This is a small special case of the theory for finitely-generated nilpotent
groups and, more generally, lattices in Lie groups of polynomial growth (see Pansu
and Breuillard \cite{pansu,breuillard}).
Here we give an elementary proof in terms of the combinatorial group theory and Euclidean
geometry.

\begin{lemma}\label{ineq1}
For $\W\in\Z^d$, $\normL{\W}\le |\W|$.
\end{lemma}

\begin{proof}
We prove that $\max_{\W\in S_n}\normL{\W}\le n$ for all $n\ge 1$ by induction on $n$.  If
$n=1$, this is immediate from the definition of $L$.  When $n>1$, we can always write
$\W\in S_n$ as $\W=\W'+\A$ where $\W'\in S_{n-1}$ and $\A\in S$ .  But then $\normL{\W} \le
\normL{\W'}+\normL{\A} \le n.$ \end{proof}

We say that $\V\in\Z^d$ has a \textit{simple spelling} in terms of the generating set if there is
a geodesic spelling which uses only significant generators.  We will denote by $\eP_n$ the set of
points in $S_n$ which have a simple spelling.

\begin{defn}  Fix a triangulation of $L$, and $n\ge 1$.
For $0\le k \le n$, let 
$$
	\eP_n:=\left\{ \sum a_i \V_i :
		\quad a_i ~\hbox{\rm non-negative integers},
		\;\{\V_i\} ~\hbox{\rm bound a simplex},
		\;\textstyle\sum a_i=n   \right\}
$$
be the set of non-negative integer combinations of the extreme points of simplices, with weights summing to $n$.
\end{defn}

Note that if $\P\in \eP_n$, then it lies in $nL\cap S_n$.  This is because
$\P$ belongs to the facet of $nL$ that its extreme points $\V_i$ do, so $\|\P\|_L=n$.  Its wordlength is at
most $n$ because we have a spelling of length $n$, so Lemma~\ref{ineq1} ensures that the
wordlength is equal to $n$.  Note also that each $\P\in \eP_n$ is at distance two in the word metric
from the other elements that differ by $\V_i-\V_j$, and therefore $\eP_n$ is $2$-dense in $nL$ with respect to the $L$ metric as well.

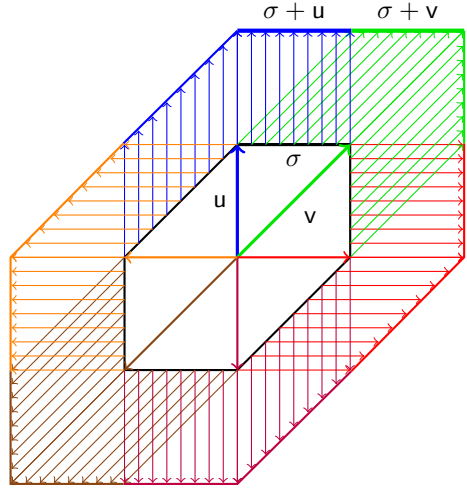
\begin{figure}[ht]
\begin{tikzpicture}[scale=1.5]

\draw [thick] (-1,0)--(-1,-1)--(0,-1)--(1,0)--(1,1)-- node [below] {$\sigma$} (0,1)--cycle;
\draw [very thick] (1,1)--(0,1);
\def\x{\y/8}
\draw [->,blue,very thick] (0,0)-- node [left,black] {$\U$} (0,1);
\draw [->,green!90!black,very thick] (0,0)-- node [below right,black] {$\V$} (1,1);
\draw [->,red,thick] (0,0)--(1,0);
\draw [->,purple,thick] (0,0)-- (0,-1);
\draw [->,brown,thick] (0,0)--  (-1,-1);
\draw [->,orange,thick] (0,0)--  (-1,0);
\foreach \y in {0,...,8}
{\draw [->,blue] (-1+\x,\x) -- +(0,1);
\draw [->,blue] (\x,1)--+(0,1);
\draw [->,green!90!black] (\x,1)--+(1,1);
\draw [->,green!90!black] (1,1-\x)--+(1,1);
\draw [->,red] (1,1-\x)--+(1,0);
\draw [->,red] (\x,\x-1)--+(1,0);
\draw [->,purple] (\x,\x-1)--+(0,-1);
\draw [->,purple] (-\x,-1)--+(0,-1);
\draw [->,brown] (-\x,-1)--+(-1,-1);
\draw [->,brown] (-1,-\x)--+(-1,-1);
\draw [->,orange] (-1,-\x)--+(-1,0);
\draw [->,orange] (-1+\x,\x) -- +(-1,0);}
\draw [ultra thick,blue] (0,2.01)-- node [above,black] {$\sigma+\U$} (1.01,2.01);
\draw [thick,blue] (-1.01,1.01)--(0,2.01);
\draw [ultra thick,green!90!black]  (1.01,2.01)-- node [above,black] {$\sigma+\V$} (2.01,2.01);
\draw [thick,green!90!black]  (2.01,2.01)--(2.01,1.01);
\draw [thick, red] (2.01,1.01)--(2.01,0)--(1.01,-1.01);
\draw [thick, purple] (1.01,-1.01)--(0,-2.01)--(-1.01,-2.01);
\draw [thick, brown] (-1.01,-2.01)--(-2.01,-2.01)--(-2.01,-1.01);
\draw [thick, orange] (-2.01,-1.01) -- (-2.01,0)--(-1.01,1.01);
\end{tikzpicture}
\caption{This figure shows that $Q+S=2Q$.}\label{tiling-figure}
\end{figure}

\begin{lemma}[Tiles]\label{tiling}
For any whole number $n\ge 2$, \quad $(n-1)Q+S=nQ.$
\end{lemma}

\begin{proof}
We know that $(n-1)Q+Q=nQ$, by convexity of $Q$.  Now consider $(n-1)Q+S$.  This contains $nS$, the extreme points of $nQ$.  
Consider $\sigma$, a simplex in the (fixed) triangulation of $L$.  For a particular $i$, the set $(n-1)\sigma+\V_i$ is a convex set (a copy of $(n-1)\sigma$).  
If the extreme points of $\sigma$ are $\{\V_i\}$, then those translates overlap, covering $n\sigma$.  Thus 
$(n-1)L+S$ includes $nL$.
(See Figure~\ref{tiling-figure}.)  But since each cone $\hat \sigma$ is contained in $Q$, one can similarly cone off to obtain the desired result.
That is, 
$$(n-1)Q+S=(n-1)\hat L + S = \left(\bigcup(n-1)\hat\sigma\right) +S = \bigcup \left( (n-1)\hat\sigma + S \right) = nQ.$$
\end{proof}

\begin{lemma}[Bounded difference]\label{bounded-diff}
There is a constant $K=K(S)\ge 1$ such that for all $\W\in\Z^d$,
$$\normL{\W} \le |\W| < \normL{\W} +K.$$
That is, $S_n$ is contained in the annular region between $(n-K)L$ and $nL$.
\end{lemma}

\begin{proof}
Set \(K=\max\{|Q\cap\Z^d|\}\).  This is the largest wordlength required to fill in the convex hull
of the generators; for example, in Figure~\ref{chessknight}, we find $K=3$.
(Note that $K$ can be arbitrarily large as the generating set $S$ varies, 
but it only depends on  $S$.)

The ball of radius $K$ then contains all of $Q\cap \Z^d$, and thus by the previous lemma we have that the ball of 
radius $K+1$ contains all of $2Q$, and so on until the ball of radius $n-1$ includes all lattice points in $(n-K)Q$.  
But since $B_n$ is contained in $nQ$, 
this precisely means that the sphere is contained in $nQ\setminus (n-K)Q$, as required.
\end{proof}

\begin{figure}[ht]
\begin{center}
\begin{tikzpicture}[scale=1/2]

  \tikzstyle{S_1} = [fill=blue!20,fill opacity=0.8]
  \tikzstyle{S_2} = [fill=pink!80,fill opacity=0.8]
    \tikzstyle{S_3} = [fill=green!80,fill opacity=0.8]
      \tikzstyle{S_4} = [fill=blue!20,fill opacity=0.8]

\foreach \a / \b in {3/10,2/20,1/30}
{\begin{scope}[scale=\a]
\draw [fill=gray!\b] (2,1)-- (1,2)-- (-1,2)--  (-2,1)--  (-2,-1)-- (-1,-2) --(1,-2)-- (2,-1)-- cycle;
\end{scope}}

\draw[fill=black] (0,0) circle (.2);

\foreach\b in {0,90,180,270}
{\begin{scope}[rotate=\b]
\foreach \a in { (1,2), (2,1)}
\draw[S_1] \a circle (.2);

\foreach \a  in {(1,1),(3,3),(2,0),(4,0),(3,1),(4,2),(1,3),(2,4)}
\draw[S_2] \a circle (.16);

\foreach \a  in {(1,0),(3,0),(5,0),(4,1),(6,1),(3,2),(5,2),(2,3),(4,3),(6,3),(1,4),(3,4),(5,4),
(2,5),(4,5),(1,6),(3,6)}
\draw[S_3] \a +(-.1,-.1) rectangle +(.1,.1);

\node[draw,star,star points=7,star point ratio=0.2,fill=orange!80, inner sep=2]  at (2,2) {};
\end{scope}}

\end{tikzpicture}
\end{center}
\caption{The {\em chess-knight metric}, with generators $\{(\pm 2,\pm1),(\pm 1,\pm 2)\}$.
The spheres of radius 1, 2, and 3 are shown entirely, as well as the first three dilates of $Q$.  
Four points from the sphere $S_4$
are shown, marked with stars, to illustrate the difference between the norm
and the word metric:  those points have $|w|=4$ while $\|w\|_L=4/3$.
It takes three steps to fill in the lattice points in $Q$, so Lemma~{bounded-diff} shows that $|w|$ and 
$\|w\|$ never differ by more than $3$.
\label{chessknight}}
\end{figure}
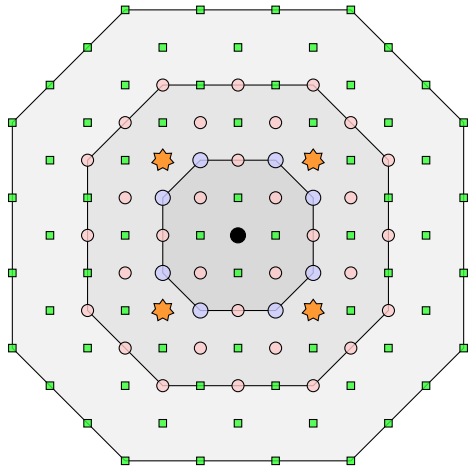

\begin{prop}[Limit shape] As a Gromov-Hausdorff limit, we have
$$\lim_{n\to\infty}{\textstyle \frac 1n} S_n=L.$$
\end{prop}

\begin{proof}
For the forward inclusion, if \(\X_n\in S_n\) is a sequence, then by the previous lemma,
$$	\lim_{n\to\infty}\left\|\frac{\X_n}{n}\right\|_L=
	\lim_{n\to\infty}\frac{\normL{\X_n}}{n} =	\lim_{n\to\infty}\frac{|\X_n|}{n} =	1.$$
The reverse inclusion follows from the fact that $\eP_n\subseteq S_n$ is $2$-dense in $nL$.
\end{proof}

\section{The limit measure}
\label{sec:measure}

We will study the measure induced on $L$ from the counting measure on  $S_n$.
First we note that the rank-one case is trivial:  considering $(\Z,S)$ with $S=-S$, 
there are exactly as many negative integers in $S_n$ as positive integers, by symmetry.
Thus the counting measure on $\frac 1n S_n$ limits to the uniform measure on the two-point set $L=\{\pm a\}$.
We study  $d\ge 2$ below.

Note that \(\frac1nS_n\cap\Delta \sigma\) is near $\sigma$ by Lemma~\ref{bounded-diff}.
Recall that if $\mu_n$, $\mu$ denote Borel probability measures on a space $X$,
then the following are
equivalent (\cite[Thm 2.4]{durrett}):
\begin{itemize}
\item For every bounded continuous $f:X\to\R$,
\[
	\lim_{n\to\infty}\int_Xfd\mu_n=\int_X f d\mu.
\]
\item For every open $U\subseteq X$,
\[
	\liminf_{n\to\infty}\mu_n(U)\ge\mu(U).
\]
\end{itemize}
In this setup, $(\mu_n)$ is said to \textit{converge weakly} to $\mu$.  On the other hand, $(\mu_n)$ is said to
\textit{converge strongly} if
\[
	\lim_{n\to\infty}\sup_{A\subseteq X}\bigl|\mu_n(A)-\mu(A)\bigr|=0.
\]

\begin{theorem}[Strong convergence on $L$]\label{strongconv}
Define measures $\mu_n$ and $\mu$  on $L$ by defining, for Lebesgue-measurable sets $\tau\subseteq L$,
$$	\mu_n(\tau) = \frac{\#(\frac1n S_n\cap \hat\tau)}{\#S_n} = \frac{\#(S_n \cap \Delta \tau)}{\# S_n}; \ \  \qquad
	\mu(\tau) = \frac{\Area(\hat\tau)}{\Area(Q)}.$$
Then $\mu_n\to\mu$ \textit{strongly}.
\end{theorem}

From this we will immediately derive the weak convergence of measures on $\R^d$ needed to prove Theorem~\ref{sphere-asymp}.
Let $\eN_r(A)$ denote the $r$-neighborhood of $A\subseteq \R^d$.
 
\begin{corollary}[Weak convergence on $\R^d$]
Define measures $\nu_n$ and $\nu$ on $\R^d$ by defining, for Lebesgue-measurable sets $A\subset \R^d$,
$$	\nu_n(A) = \frac{\#(\frac1n S_n\cap A)}{\#S_n}   ; \ \ \qquad
	\nu(A) = \mu(A\cap L).$$
Then $\nu_n\to\nu$ weakly.
\end{corollary}

\begin{proof}
Let $U\subseteq\R^d$ be open and set $\sigma=U\cap L$.
Given $\epsilon>0$, let $\sigma'\subseteq\sigma$ a closed subset
such that $\mu(\sigma')>\mu(\sigma)-\epsilon$.  Then for any metric inducing the standard topology, we can take large enough
$n$ so that $\eN_{K/n}(\sigma')\subseteq U$.  Then
$$	\liminf_{n\to\infty}\nu_n(U) \ge
	\lim_{n\to\infty}\nu_n(\hat\sigma') =
	\lim_{n\to\infty}\mu_n(\sigma') 
	=	\mu(\sigma') \ge 	\mu(\sigma)-\epsilon
	=\nu(U)-\epsilon.$$

Let $\epsilon\to 0$ to get the desired inequality.
\end{proof}

Next we demonstrate that this suffices to prove the main theorem.

\begin{proof}[Proof of Theorem~\ref{sphere-asymp}]
Suppose that $f:\Z^d\to \R$ is asymptotic to a function $g:\R^d\to\R$ that is homogeneous
of order $k$, meaning that $g(a\X)=a^kg(\X)$, and 
$\lim\limits_{\X\to\infty} \frac{f(\X)}{g(\X)}=1$.
This means that for all $\epsilon>0$ there exists $N$ such that 
$$\X\notin B_N \implies   \quad (1-\epsilon) g(\X) \le f(\X) \le (1+\epsilon) g(\X).$$
But then we have 
$$n>N \implies \quad (1-\epsilon)\sum_{\X\in S_n }  g(\X) 
\le \sum_{\X\in S_n } f(\X) \le (1+\epsilon) \sum_{\X\in S_n }g(\X),$$
which gives
$$(1-\epsilon)\frac{\sum_{\X\in S_n }  g(\X) }{n^k |S_n|}
\le \frac{\sum_{\X\in S_n } f(\X)}{n^k |S_n|} \le (1+\epsilon) \frac{\sum_{\X\in S_n }g(\X)}{n^k |S_n|}.
$$
Since $\epsilon$ was arbitrary, this means that 
$$\lim_{n\to\infty} \frac{\sum_{\X\in S_n }  f(\X) }{n^k |S_n|} =
\lim_{n\to\infty} \frac{\sum_{\X\in S_n }  g(\X) }{n^k |S_n|}.$$ 
But then 
$g(\X)/n^k= g(\X/n)$, so 
$$\frac{\sum_{\X\in S_n }  g(\X) }{n^k |S_n|} = \frac {\sum_{\X\in \frac 1n S_n} g(\X)}{|S_n|}
=\int_{\R^d} g(\X) \ d\nu_n(\X),$$ with respect to the measure defined above.
Noting that $\nu_n$ and $\nu$ are supported on the compact set $Q$,
weak convergence finishes the job:
$$\lim_{n\to\infty}    \frac{\sum_{\X\in S_n }  f(\X) }{n^k |S_n|} =
\lim_{n\to\infty}  \int_{\R^d} g(\X) \ d\nu_n(\X) = \int_{\R^d} g(\X) \ d\nu(\X)
=\int_L g(\X) \ d\mu(\X),$$ as desired.
\end{proof}

\subsection{Rank two}

We will prove Theorem~\ref{strongconv}  in this section ($d=2$) and the following section ($d>2$).
We begin by giving the necessary counting argument in $\Z^2$, 
because it has some features that are particular to that dimension.

In what follows, fix a side $\sigma$ of $L$ and let $\U$ and $\V$ be the names of its
endpoints, which are necessarily integer points.  The key step in proving
Theorem~\ref{strongconv} is to count the number of points of $S_n$ in the sector $\Delta
\sigma$ over an entire edge.  First we get control on the geodesic spellings of large
words in the sector.

\begin{lemma}[Geodesic spellings] There is a uniform bound $D_0$ such that for
sufficiently large words $\W$ in $\Delta \sigma$, there is a geodesic representative of
the form $\W=a\U+b\V+\W'$, where $|\W'|<D_0$.
\end{lemma}

\begin{proof}
Label the elements of $S$ which lie on the line segment between $\U$ and $\V$ as
$\A_1,\ldots,\A_r$.  Then the first task is to show that we can find a geodesic
representative of the form
$$\W=a\U +b\V + \sum \alpha_i\A_i + \W'',$$
where $|\W''|$ is bounded.  This is true because $\U,\V,$ and the $\A_i$ are the only
generators whose projection onto $\sigma^\perp$ (the line through the origin that is
perpendicular to $\sigma$) is one.  We know by Lemma~\ref{bounded-diff} that the
projection of $\W$ onto $\sigma^\perp$ is within $K$ of $|\W|$, and thus there is a
uniform bound on the number of other generators that can appear.

Next, write each of the $\A_i$ as $\frac{p_i}{q_i}\U + \frac{q_i-p_i}{q_i}\V$, and let
$q=\mathop{\rm lcm}\{q_i\}$.  Then without loss of generality, the coefficients $\alpha_i$
are at most $q$.  (Otherwise, $q\A_i$ can be rewritten as an integer combination of $\U$
and $\V$.)  Finally, setting $\W'=\sum \alpha_i\A_i + \W''$ completes the proof.
\end{proof}

\begin{corollary}[Modifying geodesic spellings]\label{manycopies} 
For all but boundedly many words $\W\in\Delta_n\sigma$, there is a geodesic in the Cayley graph
from the identity to $\W$ which passes through the points
$$\W, \quad  \W-\V, \quad \W-2\V, \quad \ldots \quad \W-K\V.$$
\end{corollary}

\begin{proof}  First, we show that all but boundedly many lattice points in $\Delta_n\sigma$ 
have a geodesic spelling 
in which the coefficients of $\U$ and $\V$ are each at least $K$.
We just consider the coefficient of $\V$ without loss of generality.
Note that $\normL \V=1$ by definition of the $L$-norm.  Thus words that are spelled 
$\W=a\U +b\V + \W'$
with $|\W'|<D_0$ are within $b+D_0$ of the line $\Delta \U$.  Let $D=K+D_0$.
Now note that $\eN_D(\Delta_n \U)$ has diameter $2D+1$, 
so its number of lattice points is uniformly
bounded.  This shows that $\W,\W-\V,\ldots,\W-K\V$ are all metrically between
$\W$ and $e$.

Finally, if a point is farther than $K$ from a line, then moving it by a distance $K$ will 
not cross over the line.  Thus the modified spellings still represent points in the sector
$\Delta\sigma$.
\end{proof}

Then we get a very clean result:  the integer points in $\Delta_n \sigma$ count, up to bounded additive error, the quantity we seek.

\begin{lemma}[Sphere counting  for $\Z^2$] 
\label{Z2_bijectivelemma}
$$ \# (S_n \cap \Delta \sigma ) \eadd \# (\Z^2\cap \Delta_n\sigma).$$
\end{lemma}

\begin{proof}
Let $\Phi_n: \Z^2 \to \Z^2$ be given by 
$\Phi_n(\W) = \W-m \V$, where $m=|\W|-n$.   That is, it modifies words by subtracting off copies of
$\V$ when the wordlength differs from $n$.

Now consider applying $\Phi_n$ to words $\W\in \Delta_n\sigma$.
For such words, as long as $n$ is sufficiently large and $\W$ is $D$-far from $\Delta \U$, 
Corollary~\ref{manycopies} guarantees that 
there is a geodesic representative using at least $K$ copies of $\V$.  
But we know that $0\le |\W|-n\le K$, so this 
 means that  $\Phi_n(\W) = \W-m\V$ is a point on a geodesic path from $e$ to $\W$.  Thus
 $|\Phi_n(\W)|=|\W| - (|\W|-n) = n$, or in other words $\Phi_n(\W)\in S_n$.

This argument shows that, apart from a bounded number of points, 
$\Phi_n$ gives a bijection from $\Z^2\cap \Delta_n\sigma$ to $S_n \cap \Delta \sigma$.
Injectivity follows from Corollary~\ref{manycopies}; surjectivity is established 
by noting that if $|\X|=n$ and 
$\normL{\X}=n-k$ for a point in the sector, then 
$\X+k\V\in\Delta_n\sigma$.
\end{proof}

\begin{proof}[Proof of Theorem~\ref{strongconv} when $d=2$]
The region $\Delta_n\sigma$ is a quadrilateral  with three of its four sides included, 
whose vertices have integer coordinates.
Pick's Theorem says that for any polygonal region whose extreme points are integer points, the 
area is equal to the number of integer points in the interior plus half of the integer points on the boundary
minus one ($A=i+\frac b2 -1$).  Now $\Delta_n\sigma$ contains one of the two long boundary segments, and the number of integer
points on the short boundary segments is uniformly bounded.  Therefore, up to additive error, 
the number of integer points in $\Delta_n\sigma$ is equal to its area.  But its area is exactly
$$\Area(\Delta_n\sigma) = n^2 \Area(\Delta_1\sigma)-(n-1)^2\Area(\Delta_1\sigma)=(2n-1)\Area(\hat\sigma).$$ 
Thus we have $\#(S_n \cap \Delta \sigma) \eadd (2n-1)\Area(\hat\sigma)$, and summing over all sides gives
$\#S_n = (2n-1)\Area(Q)$, 
which shows that 
$$\frac{\#(S_n \cap \Delta \sigma)}{\# S_n} \to \frac{\Area(\hat\sigma)}{\Area(Q)}.$$
(Note that this also establishes the spherical growth asymptotics for $d=2$, as in Theorem~\ref{growth}.)

To complete the proof it suffices to show that the estimate $\#(\Z^2 \cap \Delta_n \tau) \eadd \Area(\Delta_n\tau)$
is valid for small subarcs $\tau\subset \sigma$.  Consider $\Delta_n\tau$, and approximate it by an integer trapezoid $T_n$
in the following way:  for the two vertices on $nL$, replace them with nearest-possible integer vertices on $nL$, and likewise
for the two vertices on $(n-1)L$.  ($T_n$ is nondegenerate for sufficiently large $n$.)  Then it is clear that 
both the area and the number of lattice points in $T_n$ are boundedly close to those values for $\Delta_n\tau$, 
so we are done.
\end{proof}

This completes the proof that counting measure limits to cone measure on the polygon $L$.

\subsection{General rank}

For general rank $d$, we will get asymptotic comparisons rather than additive difference by carrying out the corresponding estimates.
We obtain a limit
shape $L = \lim_{n\to\infty} \frac 1n S_n$
in $\R^d$  by taking the boundary polyhedron of the  convex hull of the generators;
we have a limiting distance on $\R^d$ via the norm induced by $L$; and finally, we 
obtain a  measure on $L$ as the limit of the counting measures on $S_n$, which is proportional to the 
Euclidean volume subtended by a facet.  However, we no longer have Pick's Theorem to count the points
in regions of the facets, so we must replace that part of the argument.

Recall that the generalization to higher dimensions of Pick's Theorem is by {\em Ehrhart polynomials}:
the number of lattice points in the large dilates of a polytope with integer vertices is given by a polynomial formula in the dilation scalar.  That is, there are coefficients $\{a_i\}$ depending on $P$ such that
$$\#(\Z^d \cap nP) = (\vol P) n^d + a_{n-1}n^{d-1} + \ldots + a_0$$
for any  natural number $n$.

It follows immediately that the number of lattice points in $\Delta_n\sigma$ 
is asymptotic to $d\cdot n^{d-1}\vol(\Delta_1 \sigma)$, by letting 
$P=\Delta_1\sigma$ and noting that $\Delta_n\sigma = nP \setminus (n-1)P$.

We would like to show that the limit measure
on $L$ is uniform on each face.  First, by triangulating if necessary, we may 
assume that all the faces of $L$ are simplices with integer vertices.

\begin{lemma}[Lattice counting for $\Z^d$]
If $\sigma$ is a simplicial $(d-1)$-cell of $L$ and $\tau$ is a simplex of the same 
dimension contained in $\sigma$, then 
$$\lim_{n\to\infty}\frac{\#(\Z^d \cap \Delta_n\tau)}{\#(\Z^d\cap \Delta_n\sigma)} 
= \frac{\vol\Delta_1\tau}{\vol\Delta_1\sigma}.$$
\end{lemma}

\begin{proof}
First, observe that all the lattice points in $\Delta_n\tau$ are contained in 
finitely many dilates of $\tau$.  That is, 
$$\#(\Z^d\cap \Delta_n\tau) = 
\#\left(\Z^d \cap \bigcup_{j=1}^{q}k_j\tau\right).$$
Here, $q$ is the number of hyperplanes parallel to $\sigma$ between $\sigma$ and the 
origin which contain lattice points.  To see that this number is finite, consider the formula for the
distance from a point to a plane, remembering that $\sigma$ having integer vertices means that the hyperplane $H$
containing $\sigma$ has an equation with integer coefficients.

Let $\Lambda_j$ be the subset $\Z^d\cap k_j H$.
Note that all of the $\Lambda_j$ are translates of some common lattice $\Lambda=\Z^d\cap H'$, where 
$H'$ is the plane through the origin parallel to $H$.  Let $V$ be the covolume of $\Lambda$ and let 
$c$ be the minimal diameter of a fundamental 
domain for $\Lambda$, which exists because the set of possible diameters is discrete.
We have 
$$  (k_j-c)^{d-1}\frac{\vol\Delta_1\tau}{V}    \le  \#(\Lambda_j \cap k_j \tau)
 \le (k_j+c)^{d-1}\frac{\vol\Delta_1\tau}{V},$$
and the same inequalities with the same constants holds for $\sigma$.
But $n-1< k_j \le n$, so enlarging $c$ by one, we can write 
$$ (n-c)^{d-1} \frac{\vol\Delta_1\tau}{V} \le  \#(\Lambda_j \cap k_j \tau)
 \le (n+c)^{d-1}\frac{\vol\Delta_1\tau}{V}.$$
We can sum over $j$ and get 
$$q(n-c)^{d-1}  \frac{\vol\Delta_1\tau}{V} \le  \#(\Z^d\cap \Delta_n\tau)
 \le q(n+c)^{d-1}\frac{\vol\Delta_1\tau}{V}.$$
This means that
\[
	\#(\Z^d\cap\Delta_n\tau)
		\sim
	qn^{d-1}\frac{\vol\Delta_1\tau}{V}.
\]
Of course the same holds for $\tau=\sigma$.

Thus 
$$\lim_{n\to\infty}\frac{\#(\Z^d \cap \Delta_n\tau)}{\#(\Z^d\cap \Delta_n\sigma)} 
= \frac{\vol\Delta_1\tau}{\vol\Delta_1\sigma},$$
as required.
\end{proof}

The next difference is that Lemma \ref{Z2_bijectivelemma} no longer holds as stated, but is
replaced by an asymptotic statement.

\begin{lemma}[Sphere counting for $\Z^d$] 
$$ \# (S_n \cap \Delta \sigma ) = \# (\Z^d\cap \Delta_n\sigma) + O(n^{d-2}).$$
\end{lemma}

To prove this,  
run the same bijective argument as before on points that are outside of a $D$-neighborhood
of the cone on the boundary of $\sigma$.  
The count of points close to the boundary is clearly lower-order,
since they live in a region that measures 
length $n$ in at most $d-2$ vector directions, and is bounded in the others.

This completes the proof of Theorem~\ref{strongconv} for all $d$ (since the $d=1$ case
is elementary).

\section{Applications}\label{applications}

\subsection{Spheres versus balls}

From the sphere averages, we can quickly deduce the other averaging statement previewed in the introduction.

\begin{sphereball}[Spheres versus balls]
For any function $f:\Z^d\to \R$ that is asymptotically homogeneous of order $k$, 
$$\lim_{n\to\infty} \frac{1}{|B_n|} \sum_{\X\in B_n} \frac 1{n^k} f(\X) 
=\left(\frac{d}{d+k}\right)   \lim_{n\to\infty} \frac{1}{|S_n|} \sum_{\X\in S_n} \frac 1{n^k} f(\X) .$$

That is, the coefficients of growth for sphere averages and ball averages are related by the simple expression
$V_{g,S}= \left(\frac{d}{d+k}\right) v_{g,S}$.
\end{sphereball}

\begin{proof}
We will repeatedly use the facts that $f\sim g$ and $g(a\X)=a^kg(\X)$.  Since $\frac 1n B_n$ is uniformly filling in $Q$,
the first equality is just a Riemann sum.  
\begin{eqnarray*}
\lim_{n\to\infty} \frac{1}{|B_n|} \sum_{\X\in B_n} \frac 1{n^k}f(\X) 
&=& 
\int_Q g(\X) \ d\Vol(\X) =
\int_{I\times L} g(t\X) \ [d\cdot t^{d-1}dt] \ d\mu(\X) \\
&=& d \cdot \left(\int_0^1 t^{d+k-1} \ dt\right) \ \int_L g(\X) \ d\mu(\X) \\
&=& \left(\frac{d}{d+k}\right) \lim_{n\to\infty} \frac{1}{|S_n|} \sum_{\X\in S_n} \frac 1{n^k} f(\X) .
\end{eqnarray*}
\end{proof}

\subsection{Higher moments}

Two very natural families of asymptotically homogeneous functions on $\Z^d$ are
collection of word norms and Minkowski norms.  Applying Theorem \ref{sphereball} to these,
we get information on the expected word-norm and expected location in $\R^d$ for elements
in the ball $B_n$ of radius $n$:

\begin{corollary}[Expectations]
For any finite generating set for $\Z^d$, the expected geodesic spelling length of words in
the ball $B_n$ of radius $n$ is $\frac{d}{d+1} n$, and the expected location in $\R^n$ is
on the polygon $\frac{d}{d+1}nL$ where $L$ is the boundary of the convex hull of the chosen
generating set.
\end{corollary}

Thus,  the expected position of a word in $B_n$ is on $S_{2n/3}$, independent of the choice of generating set.

To see this in an example, consider $(\Z^2,\std)$.  As above, set up $B_n$ as the union
of $S_j$ for $0\le j\le n$, noting  that $\#S_j=4j$ for $j\ge 1$
and $\|\X\|=j$ for $\X\in S_j$.   Thus
the average wordlength over the ball is 
$$\frac{4\sum_1^n j^2}{1+4\sum_1^n j} = \frac{4n^3+6n^2+2n}{6n^2+6n+3},$$
which grows like $\frac 23 n$, as predicted.  Though it was straightforward to calculate this 
directly for the simplest choice of generators, it is not apparent a priori how to proceed for an arbitrary generating set.

More generally we can compute higher moments by applying Theorem \ref{sphereball}
to the functions $f(\X)=|\X|^p$ which are asymptotically homogeneous of order $p$:

\begin{corollary}[Higher moments]
The expected value of $|\X|^p$ over $B_n$ is $\frac{d}{d+p}$.
\end{corollary}

This tells us that 
the higher moments are independent of the choice of word metric as well.  

As we should intuitively expect, the sphere/ball theorem when $k=0$ (so $f$ is close to a 
scale-invariant function) says that 
ball-average is equal to the sphere-average.    This is also the case of  any function averaged
over a group with exponential growth, for a different reason:  there, 
almost all of the points on the ball will be concentrated
on its boundary sphere, so again the limiting ball-average is equal to the limiting sphere-average.

\subsection{Asymptotic density}
The counting results can also be used to find the density of group elements with a particular
property, $(P)$.  For instance, we can say that a group element $\W\in\Z^2$ has a 
{\em simple spelling} if $\W=a\V_i+b\V_{i+1}$ for consecutive significant generators.
We can verify that every word has a simple spelling with respect to the standard generators,
 whereas only 1 in 36 elements has a simple spelling in 
$S=\pm\{6\E_1,\E_1,6\E_2,\E_2\}$ (compare Figure~\ref{distribution}, where the words with simple spellings are those appearing
on the bounding polygon).
Let $r$ be the number of sides of the polygon $Q$ and let $A=\Area(Q)$.
There are $rn$ simple spellings of length $n$,
all representing different group elements.  
On the
other hand, $\#S_n \eadd 2An$.  Thus
$\lim\frac{\# (S_n\cap (P))}{\#S_n}= \lim\frac{\# (B_n\cap (P))}{\#B_n}=\frac{r}{2A}$, 
or in other words:

\begin{corollary} For $(\Z^2,S)$, the density of words with simple spellings is $r/2A$.
\end{corollary}

This does depend on the generating set---only on the convex hull, as usual, but not only on its area---and it holds
uniformly at large word-lengths $n$, as well as when averaging over words of length $\le n$.
As a check, 
recall that Pick's theorem says that $A=i+b/2-1$.  We know that $r\le b$ and $i\ge 1$, 
which means $r/2A \le 1$, which is required for plausibility.   Besides recovering the answers above, 
we also see for instance that 
with respect to the chess-knight generators (see Figure~\ref{chessknight}) the 
probability of simple spellings is $2/7$, which would have been extremely unpleasant
to derive by hand.

We can likewise define simple spellings in higher dimensions.  Given $(\Z^d,S)$, form $L$ as usual, and let 
$\Sigma_1,\ldots,\Sigma_k$ be the facets of $L$, regarded as subsets of $S$ via their extreme points.
(For instance, a pentagon face of some $L\subset\R^3$ corresponds to a five-element subset of $S$.)
Then a simple spelling is one of the form $$\W=\sum_{\A_i\in\Sigma_j} \alpha_i\A_i$$
for some $j$, 
with non-negative weights $\alpha_i$.  Using this and our techniques, one could obtain a number-theoretic expression for the 
proportion of simple spellings using the number of ordered partitions of large integers $n$ into 
$n_j=\# \Sigma_j$ non-negative integers.

\subsection{Sprawl and statistical hyperbolicity}
As is well known, the geometric condition called {\em hyperbolicity} (sometimes called word hyperbolicity, $\delta$-hyperbolicity, or 
Gromov hyperbolicity) gives strong algebraic and geometric information about groups.
Hyperbolicity is a large-scale invariant, so for finitely generated groups, being hyperbolic does not depend on the choice 
of (finite) generating set.  However, if we formulate a metric condition corresponding to hyperbolicity, then the measurements
themselves depend on generators.
We quantify the degree of hyperbolicity with a statistic we call the {\em sprawl} of a group (with respect to a generating set).  
We give a brief mention 
here, but develop some results and conjectures in ~\cite{dlm2}.

The sprawl of a group measures the average distance between pairs of points on the
spheres in the word metric, normalized by the radius, as the spheres get large.  Sprawl
thus gives a numerical measure of the asymptotic shape of spheres, which can be studied for arbitrary presentations
of finitely generated groups.
To be precise, let
$$E(G,S) := \lim_{n\to\infty}\frac{1}{|S_n|^2}\sum_{x,y\in S_n} \frac 1n d(x,y),$$ 
provided this limit exists.  Note that since $0\le d(x,y)\le 2n$, the value is always
between 0 and 2.  By way of interpretation, note that $E=2$ means that one can almost
always pass through the origin when traveling between any two points on the sphere without
taking a significant detour.  
This statistic is not quasi-isometry invariant but nonetheless captures
interesting features of the large-scale geometry.  

Hyperbolicity is often characterized with the slogan  that ``triangles are thin," meaning that the third side of a geodesic triangle
must stay within bounded distance of the other two sides.  In terms of $x,y\in S_n$, this says that the geodesic
$\overline{xy}$ should be about as long as $d(x,0)+d(0,y)$, provided that $\overline{0x}$ and $\overline{0y}$ 
do not fellow-travel.  Thus, if fellow-traveling is relatively rare in a Gromov hyperbolic space, 
then we will have $E=2$.
We show in \cite{dlm2} that for any non-elementary hyperbolic group 
with any generating set, $E(G,S)=2$.  
(Recall that a hyperbolic group is called {\em elementary} if it is finite or has a
finite-index cyclic subgroup.)  Thus, $E<2$ is an obstruction to hyperbolicity.  We will say that a presentation $(G,S)$ 
is {\em statistically hyperbolic} if $E(G,S)=2$; this does not imply that $G$ is a hyperbolic group, but only that 
this metric calculation works out {\em on average} as though it were.  (For example, $F_2\times \Z$ is statistically
hyperbolic with respect to its standard generators.)

We can study statistical hyperbolicity for free abelian groups with the tools developed in this paper.
For a function of several variables  $f: (\Z^d)^m\to \R$ asymptotic to 
$g:(\R^d)^m \to \R$ with $g$  homogeneous of order $k$, Theorem~\ref{sphere-asymp} tells us that 
$$\lim_{n\to\infty} \frac{1}{|S_n|^m} \sum_{\X\in (S_n)^m} \frac 1{n^k} f(\X) 
= \int_{L^m} g(\X) \ d\mu^m(\X).$$

Thus it follows immediately from the main result of this paper that 
$$E(\Z^d,S)= \int_{L^2}  \normL{\X - \Y} \ d\mu^2$$
for all finite generating sets $S$.  That means that we know exactly how sprawl depends on the generators.  (Furthermore,
there is an exact algorithm for computing this, presented in \cite{dlm2}.)  
With the usual tools for coarse geometry, $\Z^d$ would be indistinguishable from the Euclidean space $\R^d$, since they 
are quasi-isometric.  But here we can compare the geometry of $\R^3$   to $\Z^3$ with
the six standard generators $\langle \pm \E_1, \pm \E_2, \pm \E_3 \rangle$ to $\Z^3$ with the eight
generators $\langle \pm \E_1 \pm \E_2 \pm \E_3 \rangle$ to the free group $F_3$, 
and see that they are arranged from least to most hyperbolic, having $E$ values $4/3 < 7/5 < 64/45 < 2$.
Thus sprawl is indeed a tool that detects geometric differences between presentations, and allows us to 
measure their degree of hyperbolicity.

We end with a corollary which says that 
no free abelian group can ever be fully statistically hyperbolic.

\begin{corollary}
For any free abelian group $\Z^d$ with any finite generating set $S$, the sprawl statistic $E(\Z^d,S)$ is well-defined
and $E(\Z^d,S)<2$.
\end{corollary}

\begin{proof}
The sprawl is computed by integrating against a measure that is absolutely continuous with Lebesgue measure, 
and we are integrating a function whose maximum value is $2$.  But a small neighborhood $A$ of 
any point on the polyhedron has positive measure, so $A\times A$ has positive measure, and
on that set the integrand is strictly less than two.
\end{proof}


\end{document}